\documentclass[3p,longnamesfirst,sort&compress]{elsarticle}

\usepackage{amsmath,amsthm,amssymb,mathrsfs}
\usepackage{thmtools,thm-restate}
\usepackage{enumerate}
\usepackage{slashed}
\usepackage{txfonts,dsfont}
\usepackage[breaklinks,colorlinks]{hyperref}
\usepackage{nameref}
\usepackage{cleveref}

\makeatletter
\newcommand{\autorefcheckize}[1]{%
  \expandafter\let\csname @@\string#1\endcsname#1%
  \expandafter\DeclareRobustCommand\csname relax\string#1\endcsname[1]{%
    \csname @@\string#1\endcsname{##1}\wrtusdrf{##1}}%
  \expandafter\let\expandafter#1\csname relax\string#1\endcsname
}
\makeatother

\declaretheorem[numberwithin=section]{theorem}
\declaretheorem[sibling=theorem, name=Lemma]{lem}
\declaretheorem[sibling=theorem, name=Corollary]{cor}

\declaretheorem[numberwithin=section, name=Example]{eg}

\numberwithin{equation}{section}

\newcommand{\abs}[1]{\left\lvert#1\right\rvert}
\newcommand{\set}[1]{\left\{#1\right\}}

\newcommand*{\To}{\longrightarrow}
\newcommand*{\Rmn}[1]{\uppercase\expandafter{\romannumeral#1}}
\newcommand*{\dif}{\mathop{}\!\mathrm{d}}

\journal{XXX}

\allowdisplaybreaks

\begin{document}

\begin{frontmatter}

\title{Sinh-Gordon equations on finite graphs\tnoteref{s}}

\author[xtu]{Linlin Sun
}
    \ead{sunll@xtu.edu.cn}

    \address[xtu]{School of Mathematics and Computational Science, Xiangtan University, Xiangtan 411105, China}

\tnotetext[s]{The author acknowledges partial support from LMNS at Fudan University, where this work was conducted during a visit in May and June 2024. Gratitude is extended to the outstanding research environment provided by Fudan University. Appreciation is also expressed to Prof. Yang Yunyan and Prof. Hua Bobo for their valuable suggestions and support. }

\begin{abstract}
In this paper, we focus on the sinh-Gordon equation on  graphs. We introduce a uniform a priori estimate to define the topological degree for this equation with nonzero prescribed functions on finite, connected and symmetric graphs. Furthermore, we calculate this topological degree case by case and show several existence results. In particular, we prove that the classical sinh-Gordon equation with nonzero prescribed function is always solvable on such graphs. 
\end{abstract}


\end{frontmatter}


\section{Introduction}

The classical elliptic sinh-Gordon equation on a Riemann surface $\Sigma$ reads as
\begin{align*}
    -\Delta_{\Sigma}u=\sinh u
\end{align*}
which plays a very important role in the study of the construction of constant mean curvature surfaces initiated by Wente \cite{Wen86counterexample}. We refer the reader to \cite{JosWanYeZho08blow} for more details. The sinh-Gordon equation has many applications in geometric analysis, for example in the work on the Wente torus \cite{Abr87constant}.

Beyond its theoretical importance on Riemann surfaces, graph analysis is crucial for various applications, including image processing and data mining. Among the numerous research avenues, studying partial differential equations from geometry or physics on graphs is particularly promising, as highlighted in references \cite{LiVraWanYao24hypersurfaces,HuaLiWan23class,LinWu17existence,ChoLiZho19discrete,GriLinYan16kazdan,GriLinYan16yamabe,GriLinYan17existence,HanShaZha20existence,Ge18kazdan,GuHuaSun23semi-linear,HaoSun23sharp,HuaXu23existence,LinYan21heat,LiuYan20multiple,ZhaZha18convergence,HuaWanYan21mean,HuaLinYau20existence} and related works.

In this paper,  we examine sinh-Gordon equations, a family of nonlinear equations, on a finite, connected and symmetric graph $G=(V,E)$
\begin{align}\label{eq:sinh-gordon}
     -\Delta u=h_{+}e^{u}+h_{-}e^{-u}-c
 \end{align}
 where $h_+$ and $h_-$ are two real prescribed functions defined on the graph and $c$ is a real number.  This family includes notable equations such as the Kazdan-Warner equation:
\begin{align}\label{eq:KW}
     -\Delta u=he^{u}-c
 \end{align}
 and the classical sinh-Gordon equation:
 \begin{align}\label{eq:sinh-Gordon-classical}
     -\Delta u=h \sinh u-c,
 \end{align}
 where $h$ is a real prescribed function defined on the graph. 

The Kazdan-Warner equation \eqref{eq:KW} on graphs has been explored by several mathematicians. For instance, Grigor’yan, Lin, and Yang  \cite[Theorems 1-3]{GriLinYan16kazdan} completely solved the solvability problem of the  Kazdan-Warner equation on a finite, connected and symmetric graph by utilizing the variational method. That is, they derived a discrete analog of the results by Kazdan and Warner \cite{KazWar74curvature}:
 \begin{itemize}
     \item When $c=0$, a solution to  \eqref{eq:KW} exists if and only if $h\equiv0$ or $h$ changes sign and $\int_{V}h\dif\mu<0$.
     \item If $c>0$, then \eqref{eq:KW} is solvable precisely when $\max_{V}h>0$.
     \item  For $c<0$, solvability to \eqref{eq:KW} requires  $\int_Vh\dif \mu<0$, and there exists a constant  $c_{h}\in[-\infty,0)$ depending on $h$ such that \eqref{eq:KW} is solvable for   $c\in(c_h,0)$, but  unsolvable for any $c<c_h$. 
 \end{itemize}
 The degree theory has proven to be an effective approach in studying partial differential equations in Euclidean spaces or Riemannian manifolds, see for example \cite{ChaYan87prescribing}. Drawing from the foundational work of Chen and Lin \cite{CheLin03topological} on mean field equations over closed Riemann surfaces, the author, in collaboration with Wang \cite{SunWan22brouwer}, innovatively employed the degree theory to resolve the Kazdan-Warner equation on finite graphs. This groundbreaking method was later broadened by Li, the author, and Yang \cite{LiSunYan24topological} to encompass Chern-Simons Higgs models on finite graphs. Very recently, Hou and Qiao \cite{houQia24solutions} also used the degree theory to extend the result obtained by \cite{LiSunYan24topological}.

The author and Wang \cite{SunWan22brouwer} demonstrated that every solution to the Kazdan-Warner equation \eqref{eq:KW} on a finite, connected and symmetric graph is uniformly bounded when the prescribed function $h\neq0$, i.e., $h$ is a nonzero function. Consequently, one can define the topological degree $D_{h,c}$ for \eqref{eq:KW} as follows
\begin{align*}
D_{h,c}=\lim_{R\To\infty}\deg\left(-\Delta-he^{(\cdot)}+c, B_{R}^{L^{\infty}(V)},0\right).
\end{align*}
Furthermore, under the assumption $h\neq0$, they computed the degree case by case and found (see \ref{sec:KW} for details):
\begin{align*}
    D_{h,c}=\begin{cases}
        -1,& c>0,\ \max_{V}h>0;\\
        -1,&c=0,\ \bar h<0<\max_{V}h;\\
        1,&c<0,\ \min_{V}h<\max_{V}h\leq0;\\
        0,&else.
   \end{cases}
    \end{align*}

 Zhu \cite{Zhu22mean}  examined the mean field equation associated with equilibrium turbulence and the Toda system on finite, connected and symmetric graphs. Specifically, he obtained that the equation
 \begin{align*}
     -\Delta u=c_1\left(\dfrac{h_1e^{u}}{\int_{V}h_1e^{u}\dif\mu}-\psi\right)-c_2\left(\dfrac{h_2e^{-u}}{\int_{V}h_2e^{-u}\dif\mu}-\psi\right)
 \end{align*}
 admits at least one solution provided that  the prescribed functions $h_1$ and $h_2$ are nonnegative, and the constants $c_1$ and $c_2$ are positive, and the function $\psi$ satisfies $\int_{V}\psi\dif\mu=1$.

The aim of this paper is to employ topological methods to investigate the sinh-Gordon equation \eqref{eq:sinh-gordon} on graphs. We will demonstrate that every solution to the sinh-Gordon equation \eqref{eq:sinh-gordon} on a finite, connected and symmetric graph is uniformly bounded, provided  $h_+\neq0$ and $h_-\neq0$.  This ensures that the topological degree $d_{h_+,h_-,c}$ for \eqref{eq:sinh-gordon} can be defined explicitly. Furthermore, we will derive the exact formula for the topological degree $d_{h_+,h_-,c}$. Consequently, we will apply the degree theory to establish existence criteria for \eqref{eq:sinh-gordon}.

The remaining part of this paper is briefly organized as follows. In \autoref{sec:main}, we introduce key concepts related to graphs and present our principal results. In \autoref{sec:preliminary}, we review fundamental properties of functions on finite, connected and symmetric graphs. In \autoref{sec:estimate}, we obtain a uniform a priori estimate, which gives a proof of the first main result \autoref{thm:main1}. The computation of the topological degree for the sinh-Gordon equation is given in \autoref{sec:degree}, offering a  proof of the second main result \autoref{thm:main2}. This leads to the existence result when the degree is non-zero, see \autoref{cor:existence}. Finally, in \autoref{sec:existence}, we discuss the existence result for the sinh-Gordon equation, which yields a proof of the third main result \autoref{thm:main3}.

Throughout the paper, we only consider finite,  connected and symmetric graphs. Moreover, we do not distinguish between sequences and subsequences unless necessary. Additionally, the capital letter $C$  denotes uniform constants that are independent of specific solutions and may vary in different situation.

\section{Main results}\label{sec:main}
Let $G=(V,E)$ be a finite, connected and symmetric graph with vertex set $V$ and  edge set $E$. The edges may be weighted. We define a weight function  $\omega:V\times V\To\mathbb{R}$ such that  $\omega_{xy}=\omega_{yx}\geq0$ and
\begin{align*}
\omega_{xy}>0\ \Longleftrightarrow\ x\sim y \Longleftrightarrow\  xy\in E.
\end{align*}
This means $\omega_{xy}>0$ precisely when $x$ and $y$ are neighboring vertices.  
The graph $G=(V,E)$ is said to be connected if, for any two vertices $x,y\in V$ there exist vertices $x_i\in V$  with $x=x_1, y=x_m$ such that
\begin{align*}
\omega_{x_ix_{i+1}}>0,\quad i=1,\dotsc,m-1.
\end{align*}
The finiteness of  $G=(V,E)$ stems from the finite number of its vertices. We introduce a positive function (vertex measure) $\mu$ on $V$ and define the ($\mu$-)Laplace operator $\Delta$ as
\begin{align}\label{mu-Laplace}
\Delta u(x)\coloneqq\dfrac{1}{\mu_x}\sum_{y\sim x}\omega_{xy}\left(u(y)-u(x)\right),\quad x\in V, \quad u\in L^{\infty}\left(V\right).
\end{align}
For any function $f\in L^{\infty}\left(V\right)$, the integral of $f$ over $V$ is defined by
\begin{align*}
\int_Vf\dif\mu\coloneqq\sum_{x\in V}f(x)\mu_x.
\end{align*}
The Green formula follows as:
\begin{align*}
\int_{V}\Delta u\cdot v\dif\mu=-\int_{V}\Gamma\left(u,v\right)\dif\mu,
\end{align*}
where $\Gamma$, the associated gradient form, is defined by 
\begin{align*}
\Gamma(u,v)(x)\coloneqq\dfrac{1}{2\mu_x}\sum_{y\in V}\omega_{xy}\left(u(y)-u(x)\right)\left(v(y)-v(x)\right).
\end{align*}
Lastly, we denote the square of the gradient  of $u$ at a vertex $x$ by
\begin{align*}
    \abs{\nabla u(x)}^2=\Gamma(u,u)(x).
\end{align*}

Firstly, we prove a uniform a priori estimate for sinh-Gordon equations. 
\begin{theorem}\label{thm:main1} 
Every solution to the sinh-Gordon equation \eqref{eq:sinh-gordon} with nonzero prescribed functions on a finite, connected and symmetric graph is uniformly bounded.
\end{theorem}

Consider  the  map $F_{h_+,h_-,c}:L^{\infty}\left(V\right)\To L^{\infty}\left(V\right)$  defined by
\begin{align*}
    F_{h_+,h_-,c}(u)=-\Delta u-h_+e^{u}-h_-e^{-u}+c.
\end{align*}
Applying the \autoref{thm:main1} we conclude that there is no solution on the boundary $\partial B^{L^{\infty}(V)}_{R}$ for $R$ large. Hence, the Brouwer degree
\begin{align*}
\deg\left(F_{h_+,h_-,c}, B^{L^{\infty}(V)}_{R}, 0\right)
\end{align*}
is well defined for $R$ large. According to the homotopy invariance, $\deg\left(F_{h_+,h_-,c}, B^{L^{\infty}(V)}_{R}, 0\right)$ is independent of $R$ when $R\to+\infty$. In particular, we define the topological degree $d_{h_+,h_-,c}$ for the sinh-Gordon equation \eqref{eq:sinh-gordon} on the finite, connected and symmetric graph $G=(V,E)$ as follows
\begin{align*}
    d_{h_+,h_-,c}=\lim_{R\to\infty}\deg\left(F_{h_+,h_-,c},B_{R}^{L^{\infty}\left(V\right)},0\right).
\end{align*}
Notice that every solutions to the equation \eqref{eq:sinh-gordon} are critical points of the following functional
 \begin{align*}
     L^{\infty}(V)\ni u\mapsto J_{h_+,h_-,c}(u)=\int_{V}\left(\dfrac12\abs{\nabla u}^2-h_+e^{u}+h_-e^{-u}+cu\right)\dif\mu.
 \end{align*}
 Thus, if $J_{h_+,h_-,c}$ is a Morse function, i.e., every critical point of $J_{h_+,h_-,c}$ is nondegenerate, then
\begin{align}\label{eq:degree-local}
    \deg\left(F_{h_+,h_-,c}, B^{L^{\infty}(V)}_{R}, 0\right)=\sum_{u\in B^{L^{\infty}(V)}_{R}, F_{h_+,h_-,c}(u)=0}\mathrm{sgn}\det\left(\dif  F_{h_+,h_-,c}(u)\right)
\end{align}
provided $\partial B^{L^{\infty}(V)}_{R}\cap F_{h_+,h_-,c}^{-1}\left(\set{0}\right)=\emptyset$. 
For more details about the Brouwer degree and its various properties we refer the reader to Chang \cite[Chapter 3]{Cha05methods}.

Secondly, we give the exact expression of the topological degree for sinh-Gordon equations. 
\begin{theorem}\label{thm:main2}
Let $G=(V,E)$ be a finite, connected and symmetric graph.     If $h_+\neq0$ and $h_-\neq0$, then
    \begin{align*}
        d_{h_+,h_-,c}=\begin{cases}
        (-1)^{\#V_0},&V_0\coloneqq\set{x\in V:h_+(x)>0}=\set{x\in V: h_-(x)<0};\\
           0,&else.
        \end{cases}
    \end{align*}
\end{theorem}

The Kronecker existence theorem suggests that at least one solution exists if the topological degree is nonzero. From this, we can derive the following existence result:
\begin{cor}\label{cor:existence}Consider a finite, connected and symmetric graph $G=(V,E)$. If $h_+\neq0, h_-\neq0$ and 
\begin{align*}
    \set{x\in V:h_+(x)>0}=\set{x\in V:h_-(x)<0},
\end{align*}
then the equation \eqref{eq:sinh-gordon} has at least one solution. Specifically, the classical sinh-Gordon equation \eqref{eq:sinh-Gordon-classical} with nonzero prescribed function admits a solution.

\end{cor}
Finally, utilizing the sub-super solution principle \ref{lem:sub-super} and the \autoref{thm:main2}, we arrive at the following existence result:
\begin{theorem}\label{thm:main3}
Consider a finite, connected and symmetric graph $G=(V,E)$. 
    Assume $\max_{V}h_+>0,\ h_-\geq0$.
    If $\int_{V}h_+\dif\mu<0$ and
    \begin{align*}
        \inf_{u\in L^{\infty}\left(V\right)}\max_{V}\left(\Delta u+h_+e^{u}+h_-e^{-u}\right)<0,
        \end{align*}
        then the sinh-Gordon equation \eqref{eq:sinh-gordon} has no solutions when $c<c^*_{h_+,h_-}$, has at least one solution for $c=c^*_{h_+,h_-}$, and has at least two solutions if $c^*_{h_+,h_-}<c<0$.
\end{theorem}

\section{Preliminaries}\label{sec:preliminary}

In this section, we present discrete versions of several key mathematical concepts, including the strong maximum principle, Kato's inequality, Hanack's inequality, and the elliptic estimate. As mentioned above, the graph $G=(V,E)$ is always assumed to be finite, connected and symmetric. 

We begin with the following strong maximum principle:
 
\begin{lem}[Strong maximum principle, cf. \cite{SunWan22brouwer}]\label{lem:maximum} If $u$ is not a constant function, then there exists $x_1\in V$ such that
\begin{align*}
u\left(x_1\right)=\max_{V}u,\quad\Delta u\left(x_1\right)<0.
\end{align*}
\end{lem}

Next, we introduce the following Kato's inequality:
\begin{lem}[Kato's inequality \cite{SunWan22brouwer}]\label{lem:kato}
\begin{align*}
\Delta u^+\geq\chi_{\set{u>0}}\Delta u.
\end{align*}
\end{lem}

 The following Hanack inequality is also useful: 
 \begin{lem}[Hanack's inequality]\label{lem:L}
 Denote by $L=\max_{V}u-u$. For every $x\sim y$, we have the following Hanack inequality
 \begin{align}\label{eq:lem-L2}
     L(y)\leq\dfrac{\mu_x}{\omega_{xy}}\left(\dfrac{\sum_{z\sim x}\omega_{zx}}{\mu_x}L(x)-\Delta u(x)\right).
 \end{align}
 Moreover, if $\max_{V}u+\min_{V}u\geq0$, then we have for every $x\in V$
 \begin{align}\label{eq:lem-L1}
     -\Delta u(x)\leq \dfrac{\sum_{z\sim x}\omega_{zx}}{\mu_x}\left(2u(x)+L(x)\right).
 \end{align}
 
 \end{lem}
 \begin{proof}
 On the one hand, notice that for every $x, z\in V$,
 \begin{align*}
     u(x)-u(z)\geq u(x)-\max_{V}u=-L(x).
 \end{align*}
 By the definition of the Laplacian \eqref{mu-Laplace}, we have for each $y\sim x$
 \begin{align*}
     -\Delta u(x)=&\dfrac{1}{\mu_x}\sum_{z\sim x}\omega_{xz}\left(u(x)-u(z)\right)\\
     =&\dfrac{\omega_{xy}}{\mu_x}\left(u(x)-u(y)\right)+\dfrac{1}{\mu_x}\sum_{z\sim x, z\neq y}\omega_{xz}\left(u(x)-u(z)\right)\\
     \geq&\dfrac{\omega_{xy}}{\mu_x}\left(u(x)-u(y)\right)-\dfrac{1}{\mu_x}\sum_{z\sim x, z\neq y}\omega_{xz}L(x)\\
     =&\dfrac{\omega_{xy}}{\mu_x}\left(L(y)-L(x)\right)-\dfrac{1}{\mu_x}\sum_{z\sim x, z\neq y}\omega_{xz}L(x)\\
     =&\dfrac{\omega_{xy}}{\mu_x}L(y)-\dfrac{1}{\mu_x}\sum_{z\sim x}\omega_{xz}L(x).
 \end{align*}
 We obtain the Hanack inequality  \eqref{eq:lem-L2}.

 On the other hand, if $\max_{V}u+\min_{V}u\geq0$, then we have for every $x, z\in V$
 \begin{align}\label{eq:L1-lem}
    u(x)-u(z)\leq u(x)-\min_{V}u\leq u(x)+\max_{V}u=2u(x)+L(x).
 \end{align}
Inserting the above estimate \eqref{eq:L1-lem} into \eqref{mu-Laplace}, we get 
 \begin{align*}
     -\Delta u(x)=\dfrac{1}{\mu_x}\sum_{y\sim x}\omega_{xy}\left(u(x)-u(y)\right)\leq\dfrac{1}{\mu_x}\sum_{y\sim x}\omega_{xy}\left(2u(x)+L(x)\right).
 \end{align*}
 We obtain the estimate \eqref{eq:lem-L1} and complete the proof.
 
 \end{proof}

Lastly, we present the following refined elliptic estimate:
\begin{lem}[Refined elliptic estimate \cite{SunWan22brouwer}]\label{lem:elliptic}
For all $u\in L^{\infty}\left(V\right)$,
\begin{align*}
\max_{V}u-\min_{V}u\leq \dfrac{A^{\#V-1}-1}{A-1}B\cdot\max_{V}\Delta u,
\end{align*}
where
\begin{align*}
    A=\max_{y\sim x}\dfrac{\sum_{z\sim x}\omega_{zx}}{\omega_{yx}}, \quad B=\max_{y\sim x}\dfrac{\mu_{x}}{\omega_{yx}}.
\end{align*}
\end{lem}
\begin{proof}

Assume 
\begin{align*}
    \max_{V}u=u(\bar x),\quad \min_{V}u=u(\underline{x}).
\end{align*}
Choose a shortest path $x_0x_1\dots x_m$ connecting $x_0=\bar x$ and $x_{m}=\underline{x}$. We consider the function $L(x)=\max_{V}u-u$. According to the \autoref{lem:L}, we get
\begin{align}\label{eq:elliptic1}
    L(x_{i+1})\leq A\cdot L(x_i)+B\cdot\max_{V}\left(-\Delta u\right),\quad i=0,1,2,\dots,m-1.
\end{align}
Since $L(x_0)=0$, we obtain by induction
\begin{align}\label{eq:elliptic2}
    L(x_i)\leq \left(1+A+\dotsm+A^{i-1}\right)B\cdot\max_{V}\left(-\Delta u\right),\quad i=1,2,\dots,m.
\end{align}
To see this, since $L(x_1)\leq B\cdot\max_{V}\left(-\Delta u\right)$ according to \eqref{eq:elliptic1}, we may assume for some $1\leq i\leq m-1$
\begin{align*}
    L(x_i)\leq \left(1+A+\dotsm+A^{i-1}\right)B\cdot\max_{V}\left(-\Delta u\right).
\end{align*}
 Applying the estimate \eqref{eq:elliptic1}, we obtain
\begin{align*}
    L(x_{i+1})\leq& A\cdot\left(\left(1+A+\dotsm+A^{i-1}\right)B\cdot\max_{V}\left(-\Delta u\right)\right)+B\cdot\max_{V}\left(-\Delta \right)\\
    =&\left(1+A+\dotsm+A^{i}\right)B\cdot\max_{V}\left(-\Delta u\right).
\end{align*}
We obtain the estimate \eqref{eq:elliptic2}. 
In particular
\begin{align*}
    \max_{V}u-\min_{V}u=L(x_{m})\leq \left(1+A+\dotsm+A^{m-1}\right)B\cdot\max_{V}\left(-\Delta u\right). 
\end{align*}
Thus
\begin{align*}
\max_{V}u-\min_{V}u\leq \left(1+A+\dotsm+A^{\#V-2}\right)B\cdot\max_{V}\left(-\Delta u\right). 
\end{align*}
Replacing $u$ by $-u$, we obtain
\begin{align*}
    \max_{V}u-\min_{V}u\leq \left(1+A+\dotsm+A^{\#V-2}\right)B\cdot\max_{V}\Delta u=\dfrac{A^{\#V-1}-1}{A-1}B\cdot\max_{V}\Delta u. 
\end{align*}
We complete the proof.

\end{proof}

\section{A priori estimates}\label{sec:estimate}
 In this section, we give a priori estimate for the sinh-Gordon type equation \eqref{eq:sinh-gordon}. That is, we will give a proof of the \autoref{thm:main1}.
 
 To elucidate the core concept of the proof, we begin by examining the following example. 

 \begin{eg}\label{eg:example-a priori}
 Let $G=(V,E)$ be a graph with only two vertices. The sinh-Gordon equation \eqref{eq:sinh-gordon} becomes
    \begin{align}\label{eq:sinh-gordon-eg}
        \begin{cases}
            \frac{\omega_{x_1x_2}}{\mu_{x_1}}\left(u(x_1)-u(x_2)\right)=h_+(x_1)e^{u(x_1)}+h_{-}(x_1)e^{-u(x_1)}-c,\\
            \frac{\omega_{x_1x_2}}{\mu_{x_2}}\left(u(x_2)-u(x_1)\right)=h_+(x_2)e^{u(x_2)}+h_{-}(x_2)e^{-u(x_2)}-c.
        \end{cases}
    \end{align}
    We assume
    \begin{align}\label{eq:assumption-eg}
        \abs{h_+(x_1)}+\abs{h_+(x_2)}>0,\quad \abs{h_-(x_1)}+\abs{h_-(x_2)}>0.
    \end{align}
Assume $u$ solves the equation \eqref{eq:sinh-gordon-eg}.     Without loss of generality, assume 
    \begin{align*}
        u(x_1)\geq \abs{u(x_2)}.
    \end{align*}
    If $h_+(x_1)>0$, then 
    \begin{align*}
        h_+(x_1)e^{u(x_1)}+h_{-}(x_1)e^{-u(x_1)}-c=\dfrac{\omega_{x_1x_2}}{\mu_{x_1}}\left(u(x_1)-u(x_2)\right)\leq\dfrac{2\omega_{x_1x_2}}{\mu_{x_1}}u(x_1)
    \end{align*}
    which implies
    \begin{align*}
        u(x_1)\leq C.
    \end{align*}
    If $h_+(x_1)\leq0$, then
    \begin{align*}
        \dfrac{\omega_{x_1x_2}}{\mu_{x_1}}\left(u(x_1)-u(x_2)\right)=h_+(x_1)e^{u(x_1)}+h_{-}(x_1)e^{-u(x_1)}-c\leq C
    \end{align*}
    which implies
    \begin{align*}
       0\leq u(x_1)-u(x_2)\leq C.
    \end{align*}
    If $u(x_2)<0$, then $u(x_1)\leq C$. If $u(x_2)\geq0$, then 
    \begin{align*}
        \abs{h_+(x_1)}+\abs{h_+(x_2)}=&\abs{\left(\dfrac{\omega_{x_1x_2}}{\mu_{x_1}}\left(u(x_1)-u(x_2)\right)-h_{-}(x_1)e^{-u(x_1)}+c\right)e^{-u(x_1)}}\\
        &+\abs{\left(\dfrac{\omega_{x_1x_2}}{\mu_{x_2}}\left(u(x_2)-u(x_1)\right)-h_{-}(x_2)e^{-u(x_2)}+c\right)e^{-u(x_2)}}\\
        \leq& Ce^{-u(x_1)},
    \end{align*}
    which implies
    \begin{align*}
        u(x_1)\leq C.
    \end{align*}
    That is, we obtain a uniform a priori estimate for the sinh-Gordon equation \eqref{eq:sinh-gordon-eg} under the assumption \eqref{eq:assumption-eg}. 
 \end{eg}
 
 Now, drawing inspiration from the \autoref{eg:example-a priori} provided above, we can present a comprehensive proof of the \autoref{thm:main1}. 
 
 \begin{proof}[Proof of the \autoref{thm:main1}]
To proceed with the proof by contradiction, we will assume that there exists a sequence of real functions $\set{u_n}$ on a finite, connected and symmetric graph $G=(V, E)$ satisfying
 \begin{align}\label{eq:sinh-gordon-n}
     -\Delta u_n=h_{+}e^{u_n}+h_{-}e^{u_n}-c,\quad\text{in}\ V,
 \end{align}
 and
 \begin{align}\label{eq:assumption-n}
     \max_{V}u_n=\max_{V}\abs{u_n}\to\infty,
 \end{align}
 as $n\to\infty$. Our goal is to show that this leads to a contradiction, thereby proving the original statement.
 According to the a priori estimate stated by Sun and Wang \cite[Theorem 2.1]{SunWan22brouwer}, we may assume $h_+\neq0$ and $h_-\neq0$.

 We consider the functions $L_n:V\To\mathbb{R}$ defined by
 \begin{align}\label{eq:L0}
     L_n=\max_V{u_n}-u_n.
 \end{align}
 It follows from \eqref{eq:assumption-n} that \begin{align*}
     \max_{V}u_n+\min_{V}u_n\geq0.
 \end{align*}
According to the \autoref{lem:L}, we obtain
\begin{align}\label{eq:L2}
     -\Delta u_n\leq C\left(u_n^++L_n\right),
 \end{align}
 and
\begin{align}\label{eq:L5}
     L_n(y)\leq C\left(\left(-\Delta u_n(x)\right)^{+}+L_n(x)\right),\quad\forall y\sim x.
 \end{align}

Insert \eqref{eq:L2} into the equation \eqref{eq:sinh-gordon-n} to obtain
 \begin{align}\label{eq:L3}
     h_+e^{u_n}+h_-e^{-u_n}\leq C\left(u_n^++L_n+1\right).
 \end{align}
   If $L_n(x)\leq C$, then by the definition \eqref{eq:L0} and the assumption \eqref{eq:assumption-n}, we get $u_n(x)\to+\infty$ as $n\to\infty$. 
The estimate \eqref{eq:L3} then gives $h_+(x)\leq0$ which implies from \eqref{eq:sinh-gordon-n} that 
 \begin{align}\label{eq:L4}
     -\Delta u_n(x)\leq h_-(x)e^{-u_n(x)}-c\leq C.
 \end{align}
 It follows from \eqref{eq:L4} and \eqref{eq:L5} that
 \begin{align*}
     L_n(y)\leq C,\quad\forall y\sim x.
 \end{align*}
 In other words, if $L_n(x)$ is uniformly bounded from above, then $L_n(y)$ is also uniformly bounded from above for every neighbors $y$ of $x$. That is 
 \begin{align}\label{eq:L6}
     L_n(x)\leq C\quad \Longrightarrow \quad  L_n(y)\leq C,\quad\forall y\sim x.
 \end{align}

Since the graph is finite, we may assume that for some points $\bar x, \underline{x}\in V$
 \begin{align*}
     u_n(\bar x)=\max_{V}u_n,\quad u_n(\underline{x})=\min_{V}u_n,\quad\forall n.
 \end{align*}
 Since the graph is connected, we can choose a shortest path $x_0x_1\dotsm x_m$ connected $x_0=\bar x$ and $x_m=\underline{x}$. We have
 \begin{align*}
     L_n(x_0)=L_n(\bar x)=0,\quad \forall n.
 \end{align*}
 Applying the estimate \eqref{eq:L6}, since $m\leq\#V-1$, we obtain by induction that
 \begin{align*}
     L_n(x)\leq C,\quad\forall x\in V.
 \end{align*}
 In particular
 \begin{align}\label{eq:osc-n}
     \max_{V}u_n-\min_{V}u_n=L_n(x_m)\leq C.
 \end{align}
 Together with the above estimate \eqref{eq:osc-n} and the assumption \eqref{eq:assumption-n}, we conclude that 
 \begin{align}\label{eq:L7}
     \lim_{n\to\infty}u_n(x)=+\infty,\quad\forall x\in V.
 \end{align}
 Moreover, by the definition of the Laplacian \eqref{mu-Laplace}, the estimate \eqref{eq:osc-n} also implies
 \begin{align}\label{eq:L8}
     \abs{\Delta u_n}\leq C\left(\max_{V}u_n-\min_{V}u_n\right)\leq C.
 \end{align}
 Combine \eqref{eq:sinh-gordon-n}, \eqref{eq:osc-n}, \eqref{eq:L7} and \eqref{eq:L8}, 
 \begin{align*}
     \abs{h_+}=e^{\max_{V}u_n-u_n}\abs{c-\Delta u_n-h_-e^{-u_n}} e^{-\max_{V}u_n}\leq Ce^{-\max_{V}u_n}
 \end{align*}
 which implies that $h_+\equiv0$. This is a contradiction and we complete the proof.

 \end{proof}

\section{Topological degrees}\label{sec:degree}
In this section, we calculate the topological degree on a case-by-case basis.

We start with a uniform a priori estimate for the sinh-Gordon equation \eqref{eq:sinh-gordon} that involves variable coefficients.

\begin{theorem}\label{thm:apriori}
 Let $G=(V,E)$ be a finite, connected and symmetric graph. Assume for some positive constant $K$, the following conditions holds:
 \begin{enumerate}[(H1)]
      \item $K^{-1}\leq\max_{V}\abs{h_{\pm}}\leq K,\  \abs{c}\leq K$, and
      \item $h_{\pm}^2\geq \pm K^{-1}h_{\pm}$.
 \end{enumerate}
 Then there is a uniform positive constant $C$ depending only on the graph such that every solution $u$ to the sinh-Gordon equation \eqref{eq:sinh-gordon} satisfies
\begin{align}\label{eq:max-uniform}
    \max_{V}\abs{u}\leq C K.
\end{align}

 \end{theorem}
 
\begin{proof}
 
Assume $u$ is a solution to the sinh-Gordon equation \eqref{eq:sinh-gordon}, i.e.,
\begin{align*}
    -\Delta u=h_+e^{u}+h_{-}e^{-u}-c.
\end{align*}
    Without loss of generality, assume
    \begin{align*}
        \max_{V}u+\min_{V}u\geq0.
    \end{align*}
    We consider the function $L=\max_{V}u-u$.  Under the assumption $(H1)$, according to \eqref{lem:L}, we obtain
    \begin{align}\label{eq:L1-uniform}
        -\Delta u\leq C\left(u^++L\right),
    \end{align}
    and
    \begin{align}\label{eq:L2-uniform}
        L(y)\leq C\left(\left(-\Delta u\right)^+(x)+L(x)\right),\quad\forall x\sim y.
    \end{align}
    
    Inserting the above estimate \eqref{eq:L1-uniform} into the sinh-Gordon equation \eqref{eq:sinh-gordon}, under the assumption $(H1)$, we have
    \begin{align}\label{eq:L3-uniform}
         h_+e^{u}+h_-e^{-u}=-\Delta u+c\leq C\left(u^++L+K\right).
    \end{align}
    If $h_+(x)>0$ and $u(x)>0$, then under the assumption $(H2)$ the above inequality \eqref{eq:L3-uniform} gives
    \begin{align*}
        h_+(x)e^{u(x)}\leq Ce^{u(x)/2}+C(L(x)+K)
    \end{align*}
    which implies
    \begin{align*}
        e^{u(x)/2}\leq\dfrac{C+\sqrt{C^2+4h_+(x)C(L(x)+K)}}{2h_+(x)}.
    \end{align*}
    We obtain
    \begin{align*}
        h_+(x)e^{u(x)}\leq \dfrac{C^2+2h_+(x)C(L(x)+K)}{h_+(x)}.
    \end{align*}
    Thus
    \begin{align}\label{eq:L4-uniform}
        h_+(x)e^{u(x)}\leq C\left(L(x)+K\right).
    \end{align}
    Notice that the above estimate \eqref{eq:L4-uniform} also holds when $h_+(x)\leq0$ or $u(x)\leq0$. Inserting \eqref{eq:L4-uniform} into \eqref{eq:sinh-gordon}, we get
    \begin{align}\label{eq:L5-uniform}
        -\Delta u(x)=h_+(x)e^{u(x)}+h_{-}(x)e^{-u(x)}-c\leq C\left(L(x)+K\right)
    \end{align}
    provided $u(x)\geq0$. It follows from \eqref{eq:L2-uniform} and \eqref{eq:L5-uniform} that
    \begin{align}\label{eq:L6-uniform}
      u(x)\geq 0\quad\Longrightarrow\quad  L(y)\leq C\left(\left(-\Delta u\right)^+(x)+L(x)\right)\leq C_0\left(L(x)+K\right),\quad\forall x\sim y.
    \end{align}
   Here the constant $C_0$ depends only on the graph.

Now we prove the a priori estimate \eqref{eq:max-uniform}.
 We may assume
 \begin{align}\label{eq:L7-uniform}
     \max_{V}\abs{u}=\max_{V}u>  \left(C_0+C_0^2+\dotsm+C_0^{\#V-1}\right)K.
 \end{align}
 Since the graph is finite, connected and symmetric, we can choose a shortest path $x_0x_1\dotsm x_m$ such that
 \begin{align*}
     u(x_0)=\max_{V}u,\quad u(x_m)=\min_{V}u.
 \end{align*}
 Then $m\leq\#V-1$ and
 \begin{align*}
     L(x_0)=\max_{V}u-u(x_0)=0.
 \end{align*}
 Since $u(x_0)\geq0$, it follows from \eqref{eq:L6-uniform} that
 \begin{align*}
     L(x_1)\leq C_0K.
 \end{align*}
 If $L(x_i)\leq  \left(C_0+C_0^2+\dotsm+C_0^i\right)K$ for some $1\leq i<m$, then
 \begin{align*}
     u(x_i)=\max_{V}u-L(x_i)>\left(C_0+C_0^2+\dotsm+C_0^{\#V-1}\right)K-\left(C_0+C_0^2+\dotsm+C_0^i\right)K\geq0,
 \end{align*}
 which implies from \eqref{eq:L6-uniform} that
 \begin{align*}
     L(x_{i+1})\leq C_0\left(\left(C_0+C^2+\dotsm+C_0^i\right)K+K\right)=\left(C_0+C_0^2+\dotsm+C_0^{i+1}\right)K.
 \end{align*}
 By using the method of induction, we conclude that
 \begin{align*}
     L(x_{i})\leq \left(C_0+C_0^2+\dotsm+C_0^i\right)K,\quad 1\leq i\leq m.
 \end{align*}
 In particular
 \begin{align*}
     \max_{V}u-\min_{V}u=L(x_m)\leq \left(C_0+C_0^2+\dotsm+C_0^m\right)K\leq \left(C_0+C_0^2+\dotsm+C_0^{\#V-1}\right)K<\max_{V}u.
 \end{align*}
 We obtain $\min_{V}u>0$. That is, under the assumption \eqref{eq:L7-uniform}, we have
 \begin{align}\label{eq:L8-uniform}
     \max_{V}u-\min_{V}u\leq \left(C_0+C_0^2+\dotsm+C_0^{\#V-1}\right)K,\quad \min_{V}u\geq0.
 \end{align}
 By the definition of the Laplacian \eqref{mu-Laplace},
 \begin{align}\label{eq:L9-uniform}
     \abs{-\Delta u(x)}=\abs{\dfrac{1}{\mu_x}\sum_{y\sim x}\omega_{xy}\left(u(x)-u(y)\right)}\leq\dfrac{1}{\mu_x}\sum_{y\sim x}\omega_{xy}\left(\max_{V}u-\min_{V}u\right)\leq C\left(\max_{V}u-\min_{V}u\right).
 \end{align}
 Inserting the estimate \eqref{eq:L8-uniform} into the above estimate \eqref{eq:L9-uniform}, we obtain
 \begin{align}\label{eq:L10-uniform}
     \abs{\Delta u}\leq C\left(C_0+C_0^2+\dotsm+C_0^{\#V-1}\right)K.
 \end{align}
Recall from the sinh-Gordon equation \eqref{eq:sinh-gordon} that
 \begin{align*}
     h_+=e^{-u}\left(c-\Delta u-h_-e^{-u}\right).
 \end{align*}
 Consequently, the assumption $(H1)$ together with the estimate \eqref{eq:L10-uniform} implies  
 \begin{align*}
     K^{-1}\leq\max_{V}\abs{h_+}\leq \left(\abs{c}+\max_{V}\abs{h_-}+\max_{V}\abs{\Delta u}\right)e^{-\min_{V}u}\leq C\left(C_0+C_0^2+\dotsm+C_0^{\#V-1}\right)Ke^{-\min_{V}u}.
 \end{align*}
 We obtain
 \begin{align}\label{eq:L11-uniform}
     \min_{V}u\leq \ln\left(C_0+C_0^2+\dotsm+C_0^{\#V-1}\right)+\ln C+\ln K.
 \end{align}
 Therefore it follows from \eqref{eq:L8-uniform}  and \eqref{eq:L11-uniform} that
 \begin{align*}
     \max_{V}\abs{u}=\max_{V}u=&\max_{V}u-\min_{V}u+\min_{V}u\\
     \leq& \left(C_0+C_0^2+\dotsm+C_0^{\#V-1}\right)K+\ln\left(C_0+C_0^2+\dotsm+C_0^{\#V-1}\right)+\ln C+\ln K.
 \end{align*}
 In particular, we obtain the desired uniform a priori estimate \eqref{eq:max-uniform} and complete the proof. 
\end{proof}

\subsection{\texorpdfstring{$h_+\leq 0, h_-\geq 0$.}{First case}}
\begin{theorem}\label{thm:first} Let $G=(V,E)$ be a finite, connected and symmetric graph.
    If $h_+\leq 0, h_-\geq 0$ and $\max_{V}\left(\abs{h_+}+\abs{h_-}\right)>0$, then
 \begin{align*}
     d_{h_+,h_-,c}=\begin{cases}
         1,&\min_{V}h_+<0,\ \max_{V}h_->0;\\
         1,&c<0,\ h_-\equiv0;\\
         1,&c>0,\ h_+\equiv0;\\
         0,&else.
     \end{cases}
 \end{align*}
\end{theorem}

\begin{proof}
By Sun and Wang's result \cite[Theorem 2.3]{SunWan22brouwer}, without loss of generality, we may assume $\min_{V}h_+<h_+\leq0$ and $\max_{V}h_->h_-\geq0$. In this case, we will prove 
    \begin{align*}
        d_{h_+,h_-,c}=1.
    \end{align*}
    
Consider the following deformations
\begin{align*}
    \quad h_{+,t}=(1-t)h_+-t,\quad h_{-,t}=(1-t)h_-+t,\quad c_t=(1-t)c,\quad t\in[0,1].
\end{align*}
Assume $u_t$ is a solution to 
\begin{align*}
    -\Delta u_t=h_{+,t}e^{u_t}+h_{-,t}e^{-u_t}-c_t,\quad t\in[0,1].
\end{align*}
One can check that for every $t\in[0,1]$,
\begin{itemize}
    \item  the assumption $(H1)$ holds, i.e.,
    \begin{align*}
    0<\min\set{\max_{V}\abs{h_{\pm}},1}\leq\max_{V}\abs{h_{\pm,t}}\leq\max\set{\max_{V}\abs{h_{\pm}},1},\quad\abs{c_t}\leq\abs{c},
    \end{align*}
    \item and the assumptions $(H2)$ holds since
    \begin{align*}
        \max_{V}h_{+,t}\leq0,\quad\min_{V}h_{-,t}\geq0.
    \end{align*}
\end{itemize}
According to the \autoref{thm:apriori}, we conclude that $\set{u_t}_{t\in[0,1]}$ is uniformly bounded. By the homotopy invariance of the topological degree, we know that
 \begin{align*}
     d_{h_+,h_-,c}=d_{h_{+,0},h_{-,0},c_0}=d_{h_{+,1},h_{-,1},c_1}=d_{-1,1,0}.
 \end{align*}
 
 Notice that the equation 
 \begin{align}\label{eq:-110}
     -\Delta u=-e^{u}+e^{-u}=-2\sinh u
 \end{align}
 has a unique constant solution $u_c\equiv0$. To see this, it suffices to prove that the solution to the  equation \eqref{eq:-110} is unique.  If $u$ and $w$ solves the equation \eqref{eq:-110}, then
   \begin{align*}
       -\Delta(u-w)=-2\sinh u+2\sinh w.
   \end{align*}
   If $u\neq w$, then $u-w$ is not a constant function since the function $t\mapsto\sinh t$ is a strictly increasing function. We obtain by applying the strong maximum principle \autoref{lem:maximum} 
   \begin{align*}
       0<-2\sinh u(x_0)+2\sinh w(x_0)
   \end{align*}
   where $x_0\in V$ satisfies
   \begin{align*}
       u(x_0)-w(x_0)=\max_{V}(u-w).
   \end{align*}
   This implies $u<w$. Similar argument yields $u>w$ which is a contradiction. Hence, by homotopy invariance, we can compute the topological degree $d_{h_+,h_-,c}$ as follows
 \begin{align*}
     d_{h_+,h_-,c}=d_{-1,1,0}=\mathrm{sgn}\det\left(-\Delta+2\cdot\mathrm{Id}\right)=1.
 \end{align*}
\end{proof}

\subsection{\texorpdfstring{$\max_{V}h_+> 0, h_-\geq0$ or $h_+\leq 0, \max_{V}h_-<0$.}{Second case}}

\begin{theorem}\label{thm:second} Let $G=(V,E)$ be a finite, connected and symmetric graph.
    If $\max_{V}h_+>0$ and $h_-\geq0$, then  
    \begin{align*}
       d_{h_+,h_-,c}=\begin{cases}
       -1,&h_-\equiv0,\ c>0;\\
       -1,&h_-\equiv0,\ \int_{V}h_+\dif\mu<0,\ c=0;\\
        0,&else.
    \end{cases} 
    \end{align*}
    
\end{theorem}
\begin{proof}
By Sun and Wang's result \cite[Theorem 2.3]{SunWan22brouwer}, without loss of generality, assume $\max_{V}h_->h_-\geq0$.
Set
\begin{align*}
    h_{+,t}=h_{+}^+-(1-t)h_{+}^-,\quad c_t=(1-t)c.
\end{align*}
Assume $u_t$ is a solution to the  the following equation
\begin{align*}
    -\Delta u_t=h_{+,t}e^{u_t}+h_{-}e^{u_t}-c_t,\quad t\in[0,1].
\end{align*}
One can check that
\begin{align*}
    0<\min\set{\max_{V}h_+^+, 1}\leq\max_{V}\abs{h_{+,t}}\leq\max\set{\max_{V}\abs{h_{+}},1},\quad \abs{c_t}\leq\abs{c},\\
    h_{+,t}(x)>0\ \Longrightarrow\quad h_+(x)>0\quad\ \Longrightarrow\quad h_{+,t}(x)=h_+(x)>0,\\
    \min_{V}h_-\geq0.
\end{align*}
Consequently, According to the \autoref{thm:apriori}, we conclude that $\set{u_t}_{t\in[0,1]}$ is uniformly bounded.  By the homotopy invariance of the topological degree, we know that
 \begin{align*}
     d_{h_+,h_-,c}=d_{h_{+,0},h_{-},c_0}=d_{h_{+,1},h_{-},c_1}=d_{h_{+}^+,h_-,0}.
 \end{align*}
 
One can check that $d_{h_+^+,h_-,0}=0$ since the equation 
\begin{align*}
    -\Delta u=h_{+}^+e^{u}+h_-e^{-u},
\end{align*}
has no solution. Hence
\begin{align*}
    d_{h_+,h_-,c}=d_{h_+^+,h_-,0}=0.
\end{align*}

\end{proof}

Similarly, we obtain
\begin{theorem}\label{thm:second2} Let $G=(V,E)$ be a finite, connected and symmetric graph.
    If $h_+\leq0$ and $\min_{V}h_-<0$, then
    \begin{align*}
        d_{h_+,h_-,c}=\begin{cases}
            -1,&h_+\equiv0,\  c<0;\\
            -1,&h_+\equiv0,\ \int_{V}h_-\dif\mu>0,\ c=0;\\
            0,&else.
        \end{cases}
    \end{align*}
\end{theorem}

\subsection{\texorpdfstring{$\max_{V}h_+>0, \min_{V}h_-<0$.}{Third case}}
This situation is more complicated. To obtain a deeper understanding of the potential variations in the topological degree at this time, we first examine the following example.
\begin{eg}
    Let $G=(V,E)$ be a graph with only two vertices $\set{x_1, x_2}$ and one edge $x_1x_2$. 
    Without loss of generality, assume $\mu_{x_1}=\mu_{x_2}=\omega_{x_1x_{2}}=1.$
    \begin{enumerate}[\text{Case} 1.]
     \item $h_+=(1,0), h_-=(-1,0)$. 
     
     The sinh-Gordon equation \eqref{eq:sinh-gordon-eg} is equivalent to 
\begin{align*}
    \begin{cases}
        x-y=e^{x}-e^{-x}-c,\\
        y-x=-c,
    \end{cases}
\end{align*}
which is equivalent to 
\begin{align*}
    c=e^x-e^{-x}-x+y=x-y.
\end{align*}
One can check that it has a unique solution $(x_c,y_c)=\left(\ln\left(c+\sqrt{c^2+1}\right), \ln\left(c+\sqrt{c^2+1}\right)-c\right)$. The degree is
\begin{align*}
    \mathrm{sgn}\det\begin{pmatrix}1-e^{x_c}-e^{-x_c}&-1\\
    -1&1
    \end{pmatrix}=\mathrm{sgn}\left(-e^{x_c}-e^{-x_c}\right)=-1.
\end{align*}

        \item $h_+=(1,0), h_-=(0,-1)$.
 
 The sinh-Gordon equation \eqref{eq:sinh-gordon-eg} is equivalent to 
\begin{align}\label{eq:eg1}
    \begin{cases}
        x-y=e^{x}-c,\\
        y-x=-e^{-y}-c.
    \end{cases}
\end{align}
which is equivalent to 
\begin{align*}
    c=e^x-x+y=x-y-e^{-y}.
\end{align*}
It $(x,y)$ solves the above equation \eqref{eq:eg1}, then we have
\begin{align*}
    f(x)\coloneqq2x-e^{x}=x+y-c=2y+e^{-y}=-f(-y).
\end{align*}
However, a direct computation implies $f\leq 2\ln2-2<0$. As a consequence, the above equation can not have any solution.  In particular, the topological degree must be zero.
\item $h_+=(1,1), h_-=(-1,0)$. 

The sinh-Gordon equation \eqref{eq:sinh-gordon-eg} is equivalent to
\begin{align}\label{eq:c}
\begin{cases}
     x-y=e^{x}-e^{-x}-c,\\
    y-x=e^{y}-c,
\end{cases}
\end{align}
which is equivalent to 
\begin{align*}
    c=e^{x}-e^{-x}-x+y=e^y-y+x.
\end{align*}
If $(x,y)$ solves the equation \eqref{eq:c},  then $x>y$ and $c>0$. Thus
\begin{align*}
    \ln\left(\dfrac{c}{2}+\sqrt{\dfrac{c^2}{4}+1}\right)<x=y-e^y+c\leq c-1,\\
 c-1-e^{c-1}+e^{1-c}\leq x-e^{x}+e^{-x}+c=y<\ln c.
\end{align*}
Hence, to compute the degree, we may assume $c=0$. However, when $c=0$, this equation has no solution. Consequently, the degree must be zero.

\item $h_+=(1,1), h_-=(-1,-1)$. 

The sinh-Gordon equation \eqref{eq:sinh-gordon-eg} is equivalent to
\begin{align*}
\begin{cases}
     x-y=e^{x}-e^{-x}-c,\\
    y-x=e^{y}-e^{-y}-c,
\end{cases}
\end{align*}
which is equivalent to
 \begin{align*}
     c=e^{x}-e^{-x}-x+y=e^y-e^{-y}-y+x.
 \end{align*}
 One can check that it has exactly one solution
\begin{align*}
    (x_c,y_c)=\left(\ln\left(\dfrac{c}{2}+\sqrt{\dfrac{c^2}{4}+1}\right),\ln\left(\dfrac{c}{2}+\sqrt{\dfrac{c^2}{4}+1}\right)\right).
\end{align*}
To compute the degree, we may assume $c\neq0$. Consequently, the degree is
\begin{align*}
    \mathrm{sgn}\det\left.\begin{pmatrix}
        1-e^x-e^{-x}&-1\\
    -1&1-e^{y}-e^{-y}
    \end{pmatrix}\right\vert_{x=y=\ln\left(\frac{c}{2}+\sqrt{\frac{c^2}{4}+1}\right)}=1.
\end{align*}

    \end{enumerate}
   
\end{eg}

To ensure readers have a clear knowledge of the proof, we will illustrate the process step by step, moving from simple to more complicated cases. The proof is organized into several theorems, as outlined below.

We begin with the following special and simple case. 

\begin{theorem}\label{thm:special}
Let $G=(V,E)$ be a finite, connected and symmetric graph.
    If $\min_{V}h_+>0, \max_{V}h_-<0$, then 
    \begin{align*}
        d_{h_+,h_-,c}=(-1)^{\#V}.
    \end{align*}
\end{theorem}
\begin{proof}
Let $\Lambda$ be the largest eigenvalue of $-\Delta$. We consider the following deformations
\begin{align*}
    h_{+,t}=(1-t)h_++t\Lambda, \quad h_{-,t}=(1-t)h_--t\Lambda,\quad c_t=(1-t)c,\quad t\in[0,1].
\end{align*}
Assume $u_t$ solves
\begin{align*}
    -\Delta u_t=h_{+,t}e^{u_t}+h_{-,t}e^{u_t}-c_t,\quad t\in[0,1].
\end{align*}
One can check that
\begin{align*}
    0<\min\set{\max_{V}\abs{h_{\pm}},\Lambda}\leq\max_{V}\abs{h_{\pm,t}}\leq\max\set{\max_{V}\abs{h_{\pm}},\Lambda},\quad\abs{c_t}\leq\abs{c},\\
    \min_{V}h_{+,t}\geq\min\set{\min_{V}h_{+},\Lambda}>0,\quad \max_{V}h_{-,t}\leq\max\set{\max_{V}h_-, -\Lambda}<0.
\end{align*}
According to the \autoref{thm:apriori}, we conclude that $\set{u_t}_{t\in[0,1]}$ is uniformly bounded. By the homotopy invariance of the topological degree, we know that
 \begin{align*}
     d_{h_+,h_-,c}=d_{h_{+,0},h_{-,0},c_0}=d_{h_{+,1},h_{-,1},c_1}=d_{\Lambda,-\Lambda,0}.
 \end{align*}

 We consider the following functional
 \begin{align*}
     L^{\infty}(V)\ni u\mapsto J_{\Lambda}(u)=\dfrac12\int_{V}\abs{\nabla u}^2\dif\mu-2\Lambda\int_{V}\cosh u\dif\mu.
 \end{align*}
 One can check that the Euler-Lagrangian equations for this functional are
 \begin{align*}
     -\Delta u=2\Lambda\sinh u=\Lambda e^{u}-\Lambda e^{-u}.
 \end{align*}
 In other words, $u$ solves \eqref{eq:sinh-gordon} if and only if $u$ is a critical point of the functional $J_{\Lambda}$.

 For every $u,\xi\in L^{\infty}\left(V\right)$ and $t\in\mathbb{R}$, a direct computation gives
 \begin{align*}
     \left.\dfrac{\dif^2}{\dif t^2}\right\vert_{t=0}J_{\Lambda}(u+t\xi)=&\int_{V}\abs{\nabla\xi}^2\dif\mu-2\Lambda\int_{V}\cosh u\cdot\xi^2\dif\mu.
 \end{align*}
 Since $\cosh u=\frac{e^u+e^{-u}}{2}\geq1$, we obtain
 \begin{align*}
      \left.\dfrac{\dif^2}{\dif t^2}\right\vert_{t=0}J_{\Lambda}(u+t\xi)\leq\Lambda\int_{V}\xi^2\dif\mu-2\Lambda\int_{V}\xi^2\dif\mu=-\Lambda\int_{V}\xi^2\dif\mu.
 \end{align*}
 In particular, $J_{\Lambda}$ is a strictly concave function in the finite dimensional space $L^{\infty}\left(V\right)$. On the other hand,
 \begin{align*}
     J_{\Lambda}(u)\leq& C\left(\max_{V}u-\min_{V}u\right)^2-C^{-1}\Lambda\left(\cosh\max_{V}u+\cosh\min_{V}u\right)\\
     \leq& -C^{-1}\cosh\max_{V}\abs{u}\to-\infty,
 \end{align*}
 as $\max_{V}\abs{u}\to\infty$. 
 Consequently, $J_{\Lambda}$ has exactly one critical point  $u_{\Lambda}$ which is the global maximum point of the functional  $J_{\Lambda}$, i.e., the sinh-Gordon equation \eqref{eq:sinh-gordon} admits a unique solution $u_{\Lambda}$. Notice that
 \begin{align*}
      \left.\dfrac{\dif}{\dif t}\right\vert_{t=0}J_{\Lambda}(u+t\xi)=&\int_{V}\left(-\Delta u-2\Lambda\sinh u\right)\xi\dif\mu=\int_{V}F_{\Lambda,-\Lambda,0}(u)\xi\dif\mu,\\
       \left.\dfrac{\dif^2}{\dif t^2}\right\vert_{t=0}J_{\Lambda}(u+t\xi)=&\int_{V}\left(-\Delta\xi-2\Lambda\cosh u\cdot\xi\right)\xi\dif\mu=\int_{V}\dfrac{\partial F_{\Lambda,-\Lambda,0}(u)}{\partial u}\xi^2\dif\mu.
 \end{align*}
 Since $J_{\Lambda}$ is strictly concave, we know that all of the eigenvalues of $\frac{\partial F_{\Lambda,-\Lambda,0}(u)}{\partial u}$ are negative. Hence, by the definition of the topological degree \eqref{eq:degree-local}, we get
 \begin{align*}
     d_{h_+,h_-,c}=d_{\Lambda,-\Lambda,0}=\mathrm{sgn}\det\left(-\Delta-2\Lambda\cosh u_{\Lambda}\mathrm{Id}\right)=(-1)^{\#V}.
 \end{align*}

\end{proof}

Next, we consider a slightly more general case.

\begin{theorem}\label{thm:general}
Let $G=(V,E)$ be a finite, connected and symmetric graph. If $\max_{V}h_+>0, \min_{V}h_-<0$ and
 \begin{align*}
     V_0\coloneqq\set{x\in V:h_+(x)>0}=\set{x\in V: h_-(x)<0},
 \end{align*}
 then 
    \begin{align*}
        d_{h_+,h_-,c}=(-1)^{\#V_0}.
    \end{align*}
\end{theorem}
\begin{proof}

We consider the following deformations
\begin{align*}
    h_{+,t}(x)=\begin{cases}
        (1-t)h_+(x)+t\Lambda,&h_+(x)>0,\\
        (1-t)h_+(x),&h_+(x)\leq0,
    \end{cases}\quad h_{-,t}(x)=\begin{cases}
        (1-t)h_-(x)-t\Lambda,&h_-(x)<0,\\
        (1-t)h_-(x),&h_-(x)\geq0,
    \end{cases}\quad c_t=(1-t)c,
\end{align*}
where $t\in[0,1]$ and $\Lambda$ is a positive number to be determined. Assume $u_t$ solves
\begin{align*}
    -\Delta_t u_t=h_{+,t}e^{u_t}+h_{-,t}e^{-u_t}-c_t,\quad t\in[0,1].
\end{align*}
 Notice that for all $t\in[0,1]$
\begin{itemize}
    \item 
    \begin{align*}
        0<\min\set{\max_{V}h_+^+,\max_{V}\abs{h_-^-}, \Lambda}\leq\max_{V}\abs{h_{\pm,t}}\leq\max\set{\max_{V}\abs{h_{\pm}},\Lambda},\quad \abs{c_t}\leq\abs{c},
    \end{align*}
    \item 
    \begin{align*}
        h_{+,t}(x)>0\ \Longrightarrow\quad h_{+}(x)>0\quad \Longrightarrow\quad h_{+,t}(x)\geq\min\set{h_{+}(x),\Lambda}>0,
    \end{align*}
    \item 
    \begin{align*}
        h_{-,t}(x)<0\ \Longrightarrow\quad h_{-}(x)<0\quad \Longrightarrow\quad h_{-,t}(x)\leq\max\set{h_{-}(x),-\Lambda}<0.
    \end{align*}
\end{itemize}
Consequently, according to the \autoref{thm:apriori}, we conclude that $\set{u_t}_{t\in[0,1]}$ is uniformly bounded. By the homotopy invariance of the topological degree, we know that
 \begin{align*}
     d_{h_+,h_-,c}=d_{h_{+,0},h_{-,0},c_0}=d_{h_{+,1},h_{-,1},c_1}=d_{\Lambda\chi_{V_0},-\Lambda\chi_{V_0},0}.
 \end{align*}
 
  It is well known that the following boundary value problem has a unique solution 
\begin{align*}
    \begin{cases}
        \Delta u=0,&\text{in}\ V\setminus V_0,\\
        u=\phi,&\text{in}\ V_0,
    \end{cases}
\end{align*}
for every function $\phi\in L^{\infty}\left(V_0\right)$. We then obtain a linear map $\phi\mapsto P\phi=u$. We define the linear operator $L:L^{\infty}\left(V_0\right)\To L^{\infty}\left(V_0\right)$ as follows
  \begin{align*}
      \phi\mapsto L\phi\coloneqq\left(\Delta (P\phi)\right)\vert_{V_0}
  \end{align*}
  Then $u\in L^{\infty}\left(V\right)$ solves 
  \begin{align}\label{eq:sg-p}
      -\Delta u=2\Lambda\sinh u\cdot\chi_{V_0},\quad \text{in}\ V,
  \end{align}
  if and only if $\phi=u\vert_{V_0}$ solves
\begin{align}\label{eq:sinh-gordon-positive}
     -L\phi=2\Lambda\sinh\phi,\quad \text{in}\ V_0,
\end{align}
if and only if $\phi$ is a critical point of the following functional
\begin{align*}
    \tilde J_{\Lambda}(\phi)=-\int_{V_0}\left(\dfrac12\phi L\phi+2\Lambda\cosh\phi\right)\dif\mu,\quad\forall \phi\in L^{\infty}\left(V_0\right).
\end{align*}
One can check that for every functions $\phi,\eta\in L^{\infty}\left(V_0\right)$, we have
\begin{align*}
    \left.\dfrac{\dif^2}{\dif t^2}\right\vert_{t=0}\tilde J_{\Lambda}(\phi+t\eta)=-\int_{V_0}\left(\eta L\eta+2\Lambda\cosh\phi\cdot\eta^2\right)\dif\mu.
\end{align*}
Let $\Lambda_0$ be the largest eigenvalue of the operator $-L$. One can  choose $\Lambda=\Lambda_0$. A similar argument as in the proof of \autoref{thm:special}, we obtain that the equation \eqref{eq:sinh-gordon-positive} has exactly one solution. Consequently, the sinh-Gordon equation \eqref{eq:sg-p} has exactly one solution. A direct computation then yields
\begin{align*}
    d_{h_+,h_-,c}=(-1)^{\#V_0}.
\end{align*}

\end{proof}

We now turn our attention to the last case, which stands apart from the previous two and presents a more subtle complexity.
That is, we consider the case that $\max_{V}h_+>0, \min_{V}h_-<0$ and 
    \begin{align*}
        \set{x\in V: h_+(x)>0}\neq \set{x\in V: h_-(x)<0}.
    \end{align*}
    We have the following 
\begin{theorem}\label{thm:final}
Let $G=(V,E)$ be a finite, connected and symmetric graph. 
    If $\max_{V}h_+>0, \min_{V}h_-<0$ and 
    \begin{align*}
        V_+\coloneqq\set{x\in V: h_+(x)>0}\neq \set{x\in V: h_-(x)<0}\eqqcolon V_-,
    \end{align*}
    then 
    \begin{align*}
        d_{h_+,h_-,c}=0.
    \end{align*}
\end{theorem}
\begin{proof}
Without loss of generality, assume
\begin{align*}
    V_-\setminus V_+\neq\emptyset.
\end{align*}
Let $\Lambda>1$ be a positive constant to be determined. We consider the following equation
\begin{align}\label{eq:sinh-Gordon-Lambda}
    -\Delta u=h_+e^{u}+h_-e^{-u}-\Lambda.
\end{align}
Assume $u_{\Lambda}$ solves \eqref{eq:sinh-Gordon-Lambda}. Applying the \autoref{thm:apriori}, there exists a positive constant $C$ which  depends only on the graph $G=(V,E)$ and the prescribed functions $h_{\pm}$ such that
\begin{align}\label{eq:uniform-Lambda}
    \max_{V}\abs{u_{\Lambda}}\leq C\Lambda.
\end{align}

On the one hand, since the graph is finite, one can choose $x_{\Lambda}\in V$ such that 
\begin{align*}
    u_{\Lambda}(x_{\Lambda})=\max_{V}u_{\Lambda}.
\end{align*}
If $h_+(x_{\Lambda})\leq0$, then applying the maximum principle \autoref{lem:maximum}, we get
\begin{align*}
    0\leq-\Delta u_{\Lambda}(x_{\Lambda})=h_+(x_{\Lambda})e^{u_{\Lambda}(x_{\Lambda})}+h_-(x_{\Lambda})e^{-u_{\Lambda}(x_{\Lambda})}-\Lambda\leq \max_{V}h_-e^{-u_{\Lambda}(x_{\Lambda})}-\Lambda.
\end{align*}
We must have $\max_{V}h_->0$ and conclude that
\begin{align}\label{eq:uniform-uper1}
    \max_{V}u_{\Lambda}=u_{\Lambda}(x_{\Lambda})\leq \ln\max_{V}h_--\ln\Lambda.
\end{align}
If $h_{+}(x_{\Lambda})>0$, then the estimate \eqref{eq:uniform-Lambda} gives
\begin{align*}
    \max_{V}\abs{\Delta u_{\Lambda}}\leq C\left(\max_{V}u_{\Lambda}-\min_{V}u_{\Lambda}\right)\leq C\Lambda
\end{align*}
and hence
\begin{align}\label{eq:uniform-Lambda2}
    h_+(x_{\Lambda})e^{u_{\Lambda}(x_{\Lambda})}+h_-(x_{\Lambda})e^{-u_{\Lambda}(x_{\Lambda})}-\Lambda=-\Delta u_{\Lambda}(x_{\Lambda})\leq C\Lambda.
\end{align}
The above estimate \eqref{eq:uniform-Lambda2} implies
\begin{align}\label{eq:uniform-uper2}
    \max_{V}u_{\Lambda}=u_{\Lambda}\left(x_{\Lambda}\right)\leq C+\ln\Lambda.
\end{align}
Therefore, we obtain a upper bound from \eqref{eq:uniform-uper1} and \eqref{eq:uniform-uper2} that 
 \begin{align}\label{eq:upper1-Lambda}
     \max_{V}u_{\Lambda}\leq C+\ln\Lambda.
 \end{align}

 One the other hand, choose $x_0\in V_-\setminus V_+$, i.e.,
 \begin{align*}
     h_+(x_0)\leq0,\quad h_-(x_0)<0.
 \end{align*}
 By the definition,
 \begin{align*}
     -\Delta u_{\Lambda}(x_0)=\dfrac{1}{\mu_{x_0}}\sum_{y\sim x_0}\omega_{x_0y}\left(u_{\Lambda}(x_0)-u_{\Lambda}(y)\right)\geq C\left(u_{\Lambda}(x_0)-\max_{V}u_{\Lambda}\right).
 \end{align*}
Inserting the above estimate into the equation \eqref{eq:sinh-Gordon-Lambda}, we get
\begin{align*}
    C\left(u_{\Lambda}(x_0)-\max_{V}u_{\Lambda}\right)\leq h_+(x_0)e^{u_{\Lambda}(x_0)}+h_-(x_0)e^{-u_{\Lambda}(x_0)}-\Lambda\leq h_-(x_0)e^{-u_{\Lambda}(x_0)}-\Lambda.
\end{align*}
Hence
\begin{align*}
    \max_{V}u_{\Lambda}\geq u_{\Lambda}(x_0)-\dfrac{h_-(x_0)}{C}e^{-u_{\Lambda}(x_0)}+\dfrac{\Lambda}{C}\geq C^{-1}\Lambda+1+\ln\dfrac{-h_-(x_0)}{C}.
\end{align*}
Therefore, we obtain a lower bound of $\max_{V}u_{\Lambda}$ as follows
\begin{align}\label{eq:lower1-Lambda}
    \max_{V}u_{\Lambda}\geq C^{-1}\Lambda-C.
\end{align}
It follows from \eqref{eq:upper1-Lambda} and \eqref{eq:lower1-Lambda} that
\begin{align*}
    C^{-1}\Lambda-C\leq C+\ln\Lambda
\end{align*}
which implies
\begin{align*}
    \Lambda\leq C.
\end{align*}
In other words, there exist a constant $\Lambda_0$ such that the equation \eqref{eq:sinh-Gordon-Lambda} has no solution for any $\Lambda>\Lambda_0$. In particular, 
\begin{align*}
    d_{h_+,h_-,\Lambda}=0,\quad\forall\Lambda>\Lambda_0.
\end{align*}

Finally, we consider the deformation
\begin{align*}
    c_t=(1-t)c+t(\Lambda_0+1),\quad t\in[0,1].
\end{align*}
Assume $u_t$ solves
\begin{align*}
    -\Delta u_t=h_+e^{u_t}+h_{-}e^{-u_t}-c_t.
\end{align*}
Applying the \autoref{thm:apriori}, we know that $\set{u_t}_{t\in[0,1]}$ is uniformly bounded. In particular, by the homotopy invariance of the topological degree, we have
\begin{align*}
    d_{h_+,h_-,c}=d_{h_+,h_-,c_0}=h_{h_+,h_-,c_1}=d_{h_+,h_-,\Lambda_0+1}=0.
\end{align*}

\end{proof}

Now we can prove the \autoref{thm:main2}. 
\begin{proof}[Proof of the \autoref{thm:main2}]
Combine the \autoref{thm:first}, \autoref{thm:second}, \autoref{thm:second2}, \autoref{thm:special}, \autoref{thm:general} and \autoref{thm:final}, we complete the proof.  

\end{proof}

\section{An existence result}\label{sec:existence}

We first recall with discrete version of sub-super solution principle.

Let $f:V\times\mathbb{R}\To\mathbb{R}$ be a continuous function. We consider the following equation
\begin{align}\label{eq:f}
    -\Delta u=f(\cdot,u).
\end{align}
It is easy to check that $u$ solves the above equation \eqref{eq:f} if and only if $u$ is a critical point of the following functional
\begin{align*}
    J_{f}(u)\coloneqq\int_{V}\left(\dfrac12\abs{\nabla u}^2-F(\cdot,u)\right)\dif\mu,
\end{align*}
where $\frac{\partial F}{\partial u}=f$. We say that $\phi\in L^{\infty}\left(V\right)$ is a subsolution to the equation \eqref{eq:f} if 
\begin{align*}
    -\Delta\phi\leq f(\cdot,\phi),
\end{align*}
and we say that $\psi$ is a super-solution to the equation \eqref{eq:f} if 
\begin{align*}
    -\Delta\psi\geq f(\cdot,\psi).
\end{align*}

We have the following 
\begin{lem}[Sub-super solution principle, cf. \cite{SunWan22brouwer}]\label{lem:sub-super}Assume $\phi$ and $\psi$ are the sub- and super-solutions to the equation \eqref{eq:f} respectively with $\phi\leq\psi$. Than any minimizer of the functional $J_{f}$ in $\set{u\in L^{\infty}\left(V\right): \phi\leq u\leq \psi}$ solves \eqref{eq:f}.
\end{lem}

As a consequence, we  have the following existence result.
\begin{lem}\label{lem:sub-super-1}
Assume 
\begin{align*}
    \limsup_{t\to-\infty}f(\cdot,t)=f_{\infty}(\cdot).
\end{align*}
If $\bar f_{\infty}>0$ and
\begin{align*}
    \inf_{u\in L^{\infty}\left(V\right)}\max_{V}\left(\Delta u+f(\cdot,u)\right)<0,
\end{align*}
then the equation \eqref{eq:f} has at least one solution. 
\end{lem}
\begin{proof} Let $\phi_{\infty}$ be the unique solution to the following equation
\begin{align*}
    -\Delta\phi_{\infty}=f_{\infty}-\bar f_{\infty},\quad\bar\phi_{\infty}=0.
\end{align*}
Here $\bar f=\fint_{V}f\dif\mu$ is the average of the integral of the function $f$ over $V$. 
For every constant $A$, we have
\begin{align*}
    -\Delta\left(A+\phi_{\infty}\right)-f(\cdot,A+\phi_{\infty})=f_{\infty}-\bar f_{\infty}-f(\cdot,A+\phi_{\infty}).
\end{align*}
Since $\bar f_{\infty}>0$, there exists a sequence of numbers $A_n\to-\infty$ such that
\begin{align*}
    -\Delta\left(A_n+\phi_{\infty}\right)-f(\cdot,A_n+\phi_{\infty})\leq-\dfrac12 \bar f_{\infty}<0.
\end{align*}
In other words, $A_n+\phi_{\infty}$ are sub-solutions to the equation \eqref{eq:f}. Applying the \autoref{lem:sub-super}, we conclude that the equation \eqref{eq:f} is solvable if and only if it has a super-solution. 

Since
\begin{align*}
    c_f\coloneqq\inf_{u\in L^{\infty}\left(V\right)}\max_{V}\left(\Delta u+f(\cdot,u)\right)<0,
\end{align*}
we can find a function $\psi$ satisfying
\begin{align*}
    \max_{V}\left(\Delta \psi+f(\cdot,\psi)\right)<\dfrac{c_f}{2}<0.
\end{align*}
Thus, $\psi$ is a super-solution and we complete the proof. 

\end{proof}

As a consequence, we obtain an existence result, i.e., we can give a proof of the \autoref{thm:main3}. More general, we have the following
\begin{theorem}Let $G=(V,E)$ be a finite, connected and symmetric graph. Assume $\max_{V}h_+>0,\ h_-\geq0,\  \max_{V}h_->0$.
    \begin{enumerate}[(1)]
        \item There exists a  constant $\Lambda_0$ such that \eqref{eq:sinh-gordon} has no solution when $c<\Lambda_0$.
        \item If $c\leq0$, then a necessary condition to solve \eqref{eq:sinh-gordon} is $\int_{V}h_+\dif\mu<0$.
        \item If $\int_{V}h_+\dif\mu<0$ and set
        \begin{align*}
            c^*_{h_+,h_-}\coloneqq\inf_{u\in L^{\infty}\left(V\right)}\max_{V}\left(\Delta u+h_+e^{u}+h_-e^{-u}\right),
        \end{align*}
        then $c^*_{h_+,h_-}\in\mathbb{R}$.
        Moreover, if $c^*_{h_+,h_-}<0$, then the sinh-Gordon equation \eqref{eq:sinh-gordon} has no solution if $c<c^*_{h_+,h_-}$, and has at least one solution if $c=c^*_{h_+,h_-}$, and has at least two solutions if $c^*_{h_+,h_-}<c<0$.
    \end{enumerate}
\end{theorem}
\begin{proof}
Firstly, since
 \begin{align*}
     \set{h_+>0}\neq\emptyset,\quad \set{h_-<0}=\emptyset,
 \end{align*}
 the argument in the proof of the \autoref{thm:final} implies that there exists a  constant $\Lambda_0$ such that the equation \eqref{eq:sinh-gordon} has no solution when $c<\Lambda_0$.

 Secondly, we assume $c\leq 0$. If $c=0$ and $u$ solves \eqref{eq:sinh-gordon}, then $u$ can not be a constant function and 
\begin{align*}
    0>-\int_{V}e^{-u}\Delta u\dif\mu=\int_{V}h_+\dif\mu+\int_{V}h_-e^{-2u}\dif\mu\geq\int_{V}h_+\dif\mu.
\end{align*}
Thus $\int_{V}h_+\dif\mu<0$. If $c<0$, then a similar argument implies that $\bar h_+<0$.

Thirdly, if $c<0$, then for each constant $A<\ln\frac{-c}{\max_{V}\abs{h_+}}$
\begin{align*}
    -\Delta A-h_+e^{A}-h_{-}e^{-A}+c\leq\max_{V}\abs{h_+}e^{A}+c<0.
\end{align*}
That is, the constant function $A$ is a sub-solution whenever $A<\ln\frac{-c}{\max_{V}\abs{h_+}}$. Notice that
\begin{align*}
    \inf_{u\in L^{\infty}\left(V\right)}\max_{V}\left(\Delta u+h_+e^{u}+h_-e^{-u}\right)\geq\inf_{u\in L^{\infty}\left(V\right)}\max_{V}\left(\Delta u+h_+e^{u}\right)\eqqcolon c_{h_+}.
\end{align*}
If $h_+$ changes its sign and $\bar h_+<0$, then we claim that 
\begin{align*}
    \inf_{u\in L^{\infty}\left(V\right)}\max_{V}\left(\Delta u+h_+e^{u}\right)\in(-\infty,0).
\end{align*}
To see this, on the one hand, let $\phi$ be the unique solution to 
\begin{align*}
    -\Delta\phi=h_+-\bar h_+,\quad\min_{V}\phi=0.
\end{align*}
Then 
\begin{align*}
    0\leq\phi\leq C\left(\max_{V}h_+-\bar h_+\right).
\end{align*}
For every constants $a>0$ and $b=\ln a$, we have
\begin{align*}
    \Delta\left(a\phi+b\right)+h_+e^{a\phi+b}=&a\left(\bar h_++h_+\left(e^{a\phi}-1\right)\right)\\
    \leq&a\left(\bar h_++\max_{V}h_+\left(e^{a C\left(\max_{V}h_+-\bar h_+\right)}-1\right)\right)\\
    =&\dfrac{1}{C\left(1-\bar h_+/\max_{V}h_+\right)}\left[e^{a C\left(\max_{V}h_+-\bar h_+\right)}-\left(1-\dfrac{\bar h_+}{\max_{V}h_+}\right)\right]\cdot aC\left(\max_{V}h_+-\min_{V}h_+\right).
\end{align*}
Hence
\begin{align*}
    c_{h_+}\leq&\inf_{a>0}\max_{V}\left(\Delta\left(a\phi+b\right)+h_+e^{a\phi+b}\right)\\
    \leq&\dfrac{1}{C\left(1-\bar h_+/\max_{V}h_+\right)}\inf_{t>0}\left[e^{t}-\left(1-\dfrac{\bar h_+}{\max_{V}h_+}\right)\right]t\\
    <&-\dfrac{1}{2C}\dfrac{\ln^2\left(1-\bar h_+/\max_{V}h_+\right)}{2+\ln^2\left(1-\bar h_+/\max_{V}h_+\right)}<0.
\end{align*}
On the other hand, if $c_{h_+}<c<0$, then according to the \autoref{lem:sub-super-1}, we conclude that there exists a solution $u_c$ to the equation
\begin{align*}
    -\Delta u_c=h_+e^{u_c}-c.
\end{align*}
Thus
\begin{align*}
    \Delta e^{-u_c}\geq-e^{-u_c}\Delta u_c\geq h_+-ce^{-u_c}.
\end{align*}
Applying the maximum principle \autoref{lem:maximum}, we obtain
\begin{align*}
    e^{-u_c}\leq\left(\Delta+c\right)^{-1}h_+.
\end{align*}
If $c_{h_+}=-\infty$, then we obtain by letting $c\to-\infty$ 
\begin{align*}
    \max_{V}h_-\leq0,
\end{align*}
which is a contradiction. 

Finally, we assume $c<0,  c_{h_+,h_-}^*<0$. On the one hand, if  the equation \eqref{eq:sinh-gordon} is solvable, then by the definition, we must have
\begin{align*}
    c^*_{h_+,h_-}\leq c.
\end{align*}
On the other hand, according to the \autoref{lem:sub-super-1},  for every $c\in\left(c^*_{h_+,h_-},0\right)$, the equation 
\begin{align*}
    -\Delta u=h_+e^{u}+h_-e^{-u}-c
\end{align*}
has at least one solution $u_{c}$. Letting $c\to c^*_{h_+,h_-}$, according to the \autoref{thm:apriori}, we obtain a solution $u_{c^*_{h_+,h_-}}$ to the equation
\begin{align*}
     -\Delta u=h_+e^{u}+h_-e^{-u}-c^*_{h_+,h_-}.
\end{align*}
 Moreover, if $c^*_{h_+,h_-}<c<0$, we claim that the equation \eqref{eq:sinh-gordon} has at least two solutions. We can prove this claim as follows.

Let $A<\min_{V}u_{c^*_{h_+,h_-}}$ be a subsolution to the equation \eqref{eq:sinh-gordon}. Notice that $u_{c^*_{h_+,h_-}}$ is a super-solution to the equation \eqref{eq:sinh-gordon}. Applying the sub-super solution principle \autoref{lem:sub-super}, we obtain a solution $u_{c,1}$ which minimizes the functional
\begin{align*}
   \set{u: A\leq u\leq u_{c^*_{h_+,h_-}}} \ni u\to J(u)=\int_{V}\left(\dfrac12\abs{\nabla u}^2-h_+e^{u}+h_-e^{-u}+cu\right)\dif\mu.
\end{align*}
Applying the maximum principle \autoref{lem:maximum}, we know that $A<u_{c,1}<u_{c_{h_+,h_-}^*}$. Moreover, for each $\eta\in L^{\infty}\left(V\right)$,
\begin{align*}
    \left.\dfrac{\dif^2}{\dif t^2}\right\vert_{t=0}J(u_{c,1}+t\eta)\geq0.
\end{align*}
We claim that $u_{c,1}$ is a local strict minimizer. For otherwise,  there exists a nonzero function $\xi$ such that
\begin{align*}
    -\Delta\xi=\left(h_+e^{u_{c,1}}-h_-e^{-u_{c,1}}\right)\xi,
\end{align*}
and
 \begin{align*}
     \left.\dfrac{\dif^3}{\dif t^3}\right\vert_{t=0}J(u_{c,1}+t\xi)=\left.\dfrac{\dif^4}{\dif t^4}\right\vert_{t=0}J(u_{c,1}+t\xi)=0.
 \end{align*}
 We conclude that $\xi$ is not a constant function since $h_-\geq0$ and $c<0$. Moreover,
 \begin{align*}
     0=\left.\dfrac{\dif^4}{\dif t^4}\right\vert_{t=0}J(u_{c,1}+t\xi)=\int_{V}\left(-h_+e^{u_{c,1}}+h_-e^{-u_{c,1}}\right)\xi^4\dif\mu=\int_{V}\Delta\xi\cdot \xi^3\dif\mu<0,
 \end{align*}
 which is a contradiction. Therefore, $u_{c,1}$ is a local strict minimizer. According to the \autoref{thm:final}, we know that the topological degree vanishes. As a consequence, we know that there exists another solution to the equation \eqref{eq:sinh-gordon}.

\end{proof}

\appendix
\section{The topological degree for the Kazdan-Warner equation}\label{sec:KW}
In this appendix, we consider the  Kazdan-Warner equation \eqref{eq:KW}.

We have the following
\begin{theorem}[Sun-Wang \cite{SunWan22brouwer}]Every solution to the Kazdan-Warner equation with nonzero prescribed function on a finite, connected and symmetric graph is uniform bounded. 
\end{theorem}
In fact, we can prove the following 
\begin{theorem}\label{thm:apriori-appendix}Let $G=(V,E)$ be a finite, connected and symmetric graph.
    Assume there exists a positive constant $K$ satisfying the following conditions:
    \begin{enumerate}[A)]
        \item $K^{-1}\leq\max_{V}\abs{h}\leq K,\quad \abs{c}\leq K$, and
        \item $\left(\max_{V}h\right)^2-K^{-1}\max_{V}h\geq0$, and
        \item either $\abs{c}\geq K^{-1}$ or $c\geq0$ and $\bar h\leq -K^{-1}$.
    \end{enumerate}
    There exists a positive constant $C$ depending only on $K$ and  the graph such that every solution $u$ to the Kazdan-Warner equation \eqref{eq:KW} satisfies
    \begin{align*}
        \max_{V}\abs{u}\leq C.
    \end{align*}
\end{theorem}
\begin{proof}
Notice that the condition $B)$ is different from Sun and Wang \cite[Theorem 4.2]{SunWan22brouwer}. We should modify the original proof of Sun and Wang.  For convenience of the reader, we sketch the proof here which is slightly different from Sun and Wang's original one.

Assume the statement if false. One can find a sequence of functions $u_n$ solves
\begin{align}\label{eq:KW-n}
    -\Delta u_n=h_ne^{u_n}-c_n
\end{align}
and satisfies
\begin{align}\label{eq:KW-blowup}
    \lim_{n\to\infty}\max_{V}\abs{u_n}=\infty,
\end{align}
where the functions $h_n$ and constants $c_n$ satisfy the assumptions $A)-C)$ and
\begin{align*}
    \lim_{n\to\infty}h_n=h,\quad \lim_{n\to\infty}c_n=c.
\end{align*}

We claim that $u_n$ can not uniformly bounded from above. For otherwise, $\max_{V}u_n\leq C$. The equation \eqref{eq:KW-n} implies
\begin{align*}
    \abs{\Delta u_n}\leq C.
\end{align*}
Applying the \autoref{lem:elliptic}, we get
\begin{align*}
    \max_{V}u_n-\min_{V}u_n\leq C\max_{V}\abs{\Delta u_n}\leq C.
\end{align*}
Hence, the blowup assumption \eqref{eq:KW-blowup} yields
\begin{align*}
    \lim_{n\to\infty}u_n=-\infty.
\end{align*}
Moreover, up to a subsequence, we may assume $u_n-\min_{V}u_n$ converges to $w$. 
By the equation \eqref{eq:KW-n},
\begin{align*}
    -\Delta\left(u_n-\min_{V}u_n\right)=h_ne^{u_n}-c_n,
\end{align*}
we obtain
\begin{align*}
    -\Delta w=c,\quad\min_{V}w=0,
\end{align*}
which implies $c=0$ and $w\equiv0$. According to the equation \eqref{eq:KW-n} again, we get
\begin{align*}
    0=-e^{-\min_{V}u_n}\int_{V}\Delta u_n\dif\mu=\int_{V}h_ne^{u_n-\min_{V}u_n}\dif\mu-c_ne^{-\min_{V}u_n}\abs{V}.
\end{align*}
Under the assumption $C)$, we get
\begin{align*}
    -K^{-1}\geq\bar h=\lim_{n\to\infty}\fint_{V}h_ne^{u_n-\min_{V}u_n}\dif\mu\geq \liminf_{n\to\infty}c_ne^{-\min_{V}u_n}\geq0,
\end{align*}
which is a contradiction.

We may assume 
\begin{align*}
    0<\max_{V}u_n\to +\infty
\end{align*}
as $n\to\infty$. According to the Kato inequality \autoref{lem:kato},
\begin{align*}
    -\Delta u_n^-\leq-\Delta(-u_n)^{+}\cdot\chi_{-u_n>0}=\left(c_n-h_ne^{u_n}\right)\chi_{u_n<0}\leq C.
\end{align*}
Applying the \autoref{lem:elliptic}, we get
\begin{align*}
    \max_{V}u_n^-=\max_{V}u_n^--\min_{V}u_n^-\leq C\max_{V}\left(-\Delta u_n^-\right)\leq C,
\end{align*}
which implies
\begin{align*}
u_n\geq-C.
\end{align*}

By the definition of the Laplacian \eqref{mu-Laplace}, for each $y\sim x$
\begin{align*}
    -\Delta u_n(x)=\dfrac{1}{\mu_x}\omega_{xy}\left(u_n(x)-u_n(y)\right)+\dfrac{1}{\mu_x}\sum_{z\sim x, z\neq y}\omega_{xz}\left(u_n(x)-u_n(z)\right),
\end{align*}
which gives
\begin{align*}
    u_n(y)\leq C\left(u_n^+(x)+\left(\Delta u_n\right)^+(x)\right).
\end{align*}
Consequently, if $u_n(x)\leq C$ for some point $x$, then 
\begin{align*}
    \abs{\Delta u_n(x)}=\abs{h_ne^{u_n(x)}-c_n}\leq C
\end{align*}
and we get
\begin{align*}
     u_n(y)\leq C,\quad\forall y\sim x.
\end{align*}
Thus, we obtain by induction 
\begin{align*}
    \max_{V}u_n\leq C,
\end{align*}
which is a contradiction to the blowup assumption \eqref{eq:KW-blowup}. In other words, we may assume, up to a subsequence,
\begin{align*}
    \lim_{n\to\infty}u_n(x)=+\infty,\quad\forall x\in V.
\end{align*}
If $\max_{V}h_n\geq K^{-1}$, then we may assume $h_n(x_0)=\max_{V}h_n$. Thus
\begin{align*}
    h_n(x_0)e^{u_n(x_0)}-c_n=-\Delta u_n(x_0)\leq C\left(u_n^+(x_0)+1\right)
\end{align*}
since $u_n\geq-C$. We must have $u_n(x_0)\leq C$ which is impossible.

If $\max_{V}h_n<K^{-1}$, then  by the assumption $B)$ we have $\max_{V}h_n\leq0$ which gives
\begin{align*}
    -\Delta u_n\leq C
\end{align*}
and
\begin{align*}
    \abs{\Delta u_n}\leq C.
\end{align*}
It follows from \eqref{eq:KW-n} that
\begin{align*}
    h=\lim_{n\to\infty}h_n=\lim_{n\to\infty}e^{-u_n}\left(c_n-\Delta u_n\right)=0,
\end{align*}
which is a contradiction to the assumption $A)$.

\end{proof}

Based on this a priori estimate as in the \autoref{thm:apriori-appendix}, one can define the topological degree $D_{h,c}$ for the Kazdan-Warner equation \eqref{eq:KW} as follows:
\begin{align*}
    D_{h,c}=\lim_{R\to\infty}\deg\left(K_{h,c},B_{R}^{L^{\infty}(V)},0\right)
\end{align*}
where
\begin{align*}
    K_{h,c}(u)=-\Delta u-he^{u}+c,\quad\forall u\in L^{\infty}\left(V\right).
\end{align*}

\begin{theorem}[Sun-Wang \cite{SunWan22brouwer}]Let $G=(V,E)$ be a finite, connected and symmetric graph. If $h\neq0$, then
\begin{align*}
    D_{h,c}=\begin{cases}
        -1,& c>0,\ \max_{V}h>0;\\
        -1,&c=0,\ \bar h<0<\max_{V}h;\\
        1,&c<0,\ \min_{V}h<\max_{V}h\leq0;\\
        0,&else.
   \end{cases}
    \end{align*}
\end{theorem}
\begin{proof}
    For convenience of the reader, we sketch the proof here which is slightly different from Sun and Wang's original one.

    One can check that if the Kazdan-Warner equation \eqref{eq:KW} admits a solution, 
    then we have
    \begin{itemize}
        \item if $c>0$, then $\max_{V}h>0$;
        \item if $c=0$, then either $h\equiv0$ or $h$ changes sign and $\bar h<0$;
 \item if $c<0$, then $\min_{V}h<0$.
    \end{itemize}

    Now we compute the degree case by case as follows.
    \begin{enumerate}[\text{Case} 1.]
        \item If $c>0$ and $\max_{V}h>0$, then
        \begin{align*}
            D_{h,c}=D_{h^+,c}.
        \end{align*}
        
 In fact, we consider the following deformation
        \begin{align*}
            h_{t}=(1-t)h+t, \quad c_t=(1-t)c+t\epsilon,\quad t\in[0,1],
        \end{align*}
        where $\epsilon$ is a positive number to be determined. 
        If $u_t$ solves
        \begin{align*}
            -\Delta u_t=h_{t}e^{u_t}-c_t,
        \end{align*}
        then according to  the \autoref{thm:apriori-appendix} there exists a constant $C$ such that 
        \begin{align*}
            \abs{u_t}\leq C,\quad\forall t\in[0,1].
        \end{align*}
        By the homotopy invariance of the topological degree, we know that
\begin{align*}
    D_{h,c}=D_{1,\epsilon}.
\end{align*}

Finally, assume $u$ solves
\begin{align*}
    -\Delta u=e^{u}-\epsilon.
\end{align*}
If $w\neq u$ also solves
\begin{align*}
    -\Delta w=e^{w}-\epsilon,
\end{align*}
then
\begin{align*}
    \int_{V}e^{u}\dif\mu=\int_{V}e^{w}\dif\mu=\epsilon\abs{V}.
\end{align*}
We conclude that $\min_{V}(u-w)<0<\max_{V}(u-w)$. Moreover, we have the following estimate
\begin{align*}
    e^{u}\leq C\epsilon,\quad e^{w}\leq C\epsilon.
\end{align*}
On the other hand, since
\begin{align*}
    -\Delta(u-w)=e^{u}-e^{w}\leq e^{u}\left(u-w\right),
\end{align*}
we conclude that
\begin{align*}
    \max_{V}(u-w)-\min_{V}(u-w)\leq C\max_{V}\Delta(w-u)\leq C\epsilon\max_{V}(u-w).
\end{align*}
If $0<\epsilon<\epsilon_0=1/(2C)$, then we conclude that
\begin{align*}
    0<\max_{V}(u-w)\leq2\min_{V}(u-w)<0,
\end{align*}
which is a contradiction. In other words, if $0<\epsilon<\epsilon_0$, then $u_{\epsilon}=\ln\epsilon$ is the unique solution to 
\begin{align*}
    -\Delta u=e^{u}-\epsilon.
\end{align*}
Hence
\begin{align*}
    D_{h,c}=\lim_{\epsilon\searrow 0}D_{1,\epsilon}=-1.
\end{align*}

\item If $c=0, \bar h<0$ and $\max_{V}h>0$, then
\begin{align*}
    D_{h,0}=-1.
\end{align*}
In fact, we consider the following deformation
\begin{align*}
    c_t=t,\quad t\in[0,1].
\end{align*}
Assume $u_t$ satisfies
\begin{align*}
    -\Delta u_t=he^{u_t}-c_t,\quad t\in[0,1].
\end{align*}
according to  the \autoref{thm:apriori-appendix}, $\set{u_t}_{t\in[0,1]}$ is uniformly bounded. Hence
\begin{align*}
    D_{h,0}=D_{h,1}=-1.
\end{align*}

\item If $c<0, h\leq0$ and $\min_{V}h<0$, then
\begin{align*}
    D_{h,c}=1.
\end{align*}
In fact, we consider the following deformation
\begin{align*}
    h_{t}=(1-t)h-t,\quad t\in[0,1].
\end{align*}
If $u_t$ satisfies
\begin{align*}
    -\Delta u_t=h_{t}e^{u_t}-c,
\end{align*}
then \autoref{thm:apriori-appendix} implies that  $\set{u_t}_{t\in[0,1]}$ is  uniformly bounded. As a consequence
\begin{align*}
    D_{h,c}=D_{-1,c}.
\end{align*}
Notice that $\ln(-c)$ is the unique solution to the equation
 \begin{align*}
     -\Delta u=-e^{u}-c.
 \end{align*}
 As a consequence,
\begin{align*}
    D_{h,c}=D_{-1,c}=\mathrm{sgn}\det\left(-\Delta-c\mathrm{Id}\right)=1.
\end{align*}

\item If $c<0$ and $h$ changes sign, then
\begin{align*}
    D_{h,c}=0.
\end{align*}
In fact, according to  the \autoref{thm:apriori-appendix},
 \begin{align*}
     D_{h,c}=D_{1,c}.
 \end{align*}
 Notice that there is no solution to the equation
\begin{align*}
    -\Delta u=e^{u}-c,
\end{align*}
since $c<0$.
We must have
\begin{align*}
    D_{h,c}=D_{1,c}=0.
\end{align*}
    \end{enumerate} 
    
\end{proof}



\end{document}